\documentclass[11pt]{article}
\usepackage[a4paper]{anysize}\marginsize{3.5cm}{3.5cm}{1.3cm}{2.5cm}
\pdfpagewidth=\paperwidth \pdfpageheight=\paperheight
\usepackage{amsmath, amssymb, comment}

\usepackage{pgf,tikz}
\usetikzlibrary{arrows}

\renewenvironment{abstract}{\begin{quote}\textbf{\abstractname.}}{\end{quote}}

\newtheorem{introtheorem}{Theorem}
\newtheorem{thm}{Theorem}[section]
\newtheorem{theorem}[thm]{Theorem}
\newtheorem{lemma}[thm]{Lemma}
\newtheorem{proposition}[thm]{Proposition}
\newtheorem{corollary}[thm]{Corollary}

\newcommand\mkthm[2]{\newenvironment{#1}{\begin{#2}\rm}{\end{#2}}}
 \mkthm{definition}{tdefinition}
         \mkthm{remark}{tremark}
       \mkthm{example}{texample}
     \mkthm{question}{tquestion}

\newenvironment{proof}[1][Proof]{\trivlist\item[\hskip\labelsep{\textit{#1.}}]}{\hspace*{\fill}$\Box$\endtrivlist}

\renewcommand\ge{\geqslant}  % schräge Variante
\renewcommand\le{\leqslant}  % schräge Variante
\renewcommand\P{\mathbb P}
\newcommand\C{\mathbb C}
\newcommand\Q{\mathbb Q}

\newcommand\Z{\mathbb Z}
\newcommand\N{\mathbb N}
\newcommand\CC{{\mathcal{C}}}
\newcommand\T{\mathbf T}

\newcommand\be[1][@{\;}r@{\;}c@{\;}l@{\;}l@{\;}]{$$\everymath{\displaystyle}\renewcommand\arraystretch{1.2}\begin{array}{#1}}
\newcommand\ee{\end{array}$$}

\newcommand\generated[1]{\left\langle#1\right\rangle}
\newcommand\compact{\itemsep=0cm \parskip=0cm}
\newcommand\set[1]{\left\{#1\right\}}
\newcommand\linequiv{\equiv_{\rm lin}}
\newcommand\tol{\mathop{\longrightarrow}\limits}
\newcommand\matr[1]{\left(\begin{array}{*{20}{c}} #1 \end{array}\right)}
\newcommand\inverse{^{\smash-\mkern-1mu1}}
\newcommand\eqnref[1]{(\ref{#1})}

\newcommand\grant[1]{{\renewcommand\thefootnote{}\footnotetext{#1.}}}

\newcommand\eps{\varepsilon}
\newcommand\smallmatr[2][*{20}{c}]{{\fontsize{8}{9}\arraycolsep=4pt\selectfont\left(\!\!\begin{array}{#1}#2\end{array}\!\!\right)}}
\def\nmat#1 #2 #3 #4 #5 #6 #7 #8 #9 {\smallmatr{
   #1 & #2 & #3 \\
   #4 & #5 & #6 \\
   #7 & #8 & #9
   }}

\newcommand\newop[2]{\newcommand#1{\mathop{\rm #2}\nolimits}}
\newcommand\renewop[2]{\renewcommand#1{\mathop{\rm #2}\nolimits}}
\newop\Bl{Bl}
\newop\Pic{Pic}
\newop\Fix{Fix}
\newop\PGL{PGL}
\renewop\Re{Re}
\renewop\Im{Im}
\newop\br{branch}

\newcommand\address[1]{{\renewcommand\\{,}\par\medskip #1}}
\newcommand\email[2][]{{\par \textit{E-mail address:} \texttt{#2}}}

\begin{document}

\title{The Halphen cubics of order two}
\author{Thomas Bauer, Brian Harbourne, Joaquim Ro\'e, Tomasz Szemberg}
\date{\today}
\maketitle

\thispagestyle{empty}
%\keywords{Automorphisms, abelian surfaces, Halphen cubics, Hesse arrangement, rational surfaces}
%\subclass{
%14C20, % Divisors, linear systems, invertible sheaves
%14E05, % Rational and birational maps
%14H52, % Elliptic curves
%14J26, % Rational and ruled surfaces
%14K12} % Subvarieties (of Abelian varieties)

\grant{\ \\\noindent{\it Acknowledgements}: The understanding necessary for this paper grew out of discussions that occurred in
working groups at two workshops in 2015, namely ``Recent advances in
Linear series and Newton-Okounkov bodies", February 9--14 at the University of Padua, and
``Ideals of Linear Subspaces, Their Symbolic Powers and Waring Problems", February 15--21 at the
Mathematisches Forschungsinstitut Oberwolfach.
We thank the participants in these workshops and the institutions that hosted them.
We also thank the University of Freiburg for hosting visits by Bauer and Harbourne to work
on this paper with Szemberg in person for a week in the summer of 2015. In addition,
the research of Bauer was partially supported by DFG grant BA 1559/6-1,
the research of Ro\'e was partially supported by MTM 2013-40680-P (Spanish MICINN grant)
and 2014 SGR 114 (Catalan AGAUR grant), and the research of Szemberg
was partially supported by National Science Centre, Poland, grant 2014/15/B/ST1/02197\\
\noindent{\it Keywords}: Automorphisms, abelian surfaces, Halphen cubics, Hesse arrangement, rational surfaces\\
\noindent{\it Mathematics Subject Classification (2010)}:
14C20, % Divisors, linear systems, invertible sheaves
14E05, % Rational and birational maps
14H52, % Elliptic curves
14J26, % Rational and ruled surfaces
14K12} % Subvarieties (of Abelian varieties)

\begin{abstract}
For each $m\ge 1$, Roulleau and Urz\'ua give an
implicit construction of a configuration of $4(3m^2-1)$ complex plane cubic curves.
This construction was crucial for their work on surfaces of general type.
We make this construction explicit by proving that the Roulleau-Urz\'ua configuration
consists precisely of the Halphen cubics of order $m$,
and we determine specific equations of the cubics for $m=1$
(which were known) and for $m=2$ (which are new).
\end{abstract}

%*****************************************************************************

\section*{Introduction}

   For each $n=3m$, we study certain arrangements of $\frac 43(n^2-3)$ plane cubic curves;
   each curve is isomorphic to the Fermat cubic $x^3+y^3+z^3$ (i.e., to $T=\C/\Z[\zeta]$, $\zeta=e^{2\pi i/6}$).
   For $n=3$, the eight curves in the arrangement
were known from invariant theory \cite{ADPHR93} and
give the ``inscribed and circumscribed'' cubics
\cite[Proposition 5.2]{ArtebaniDolgachev09} for the four singular cubics
in the Hesse pencil $\langle x^3+y^3+z^3, xyz\rangle$, but for $m>1$ the arrangements are 
described explicitly here for the first time.

   These arrangements come from arrangements of $4n^2$ elliptic curves on the abelian surface $T\times T$,
   studied by Hirzebruch in \cite{Hirz84}. Using a quotient by a finite group action followed by a blow up to resolve singularities,
   Roulleau and Urz\'ua obtain in \cite[section 3]{RU:Chern} corresponding arrangements of
   $\frac 43(n^2-3)$ elliptic curves on a rational surface they call $H$, whose images (see \cite[Section 1]{Roulleau14})
   under a birational morphism to $\P^2$ give the arrangements of plane cubics we study here.
   The Roulleau-Urz\'ua arrangements on $H$ were crucial for their construction of surfaces of
   general type in \cite{RU:Chern}. In \cite{Roulleau14}, Roulleau shows that the corresponding 
   arrangements of plane cubics are interesting for another reason: the Harbourne index of the 
   union $C_n$ of the $\frac 43(n^2-3)$ cubics has limit $-4$ as $n\to\infty$. We note that
   no reduced plane curve is yet known with Harbourne index less than or even equal to $-4$.

The arrangements of cubics as given in \cite{Roulleau14} are not given explicitly;
they are described only as images under birational maps.
   In this paper we construct the same arrangements
   of cubics by elementary methods directly on $\P^2$,
   which allows us to give
   explicit equations for these interesting cubics.
   In the case $n=3$, the corresponding eight cubics
   coincide with the classical \emph{Halphen cubics}
   \cite{Halphen:recherches,ArtebaniDolgachev09,ADPHR93, Fr02}. For $n=3m$ with $m>1$,
   they are members of four pencils, where each pencil is spanned by
   two of the eight Halphen cubics. We call these pencils {\it Hesse singular point cubic pencils}.
   (The reason for this name is that these pencils can also be obtained 
   from the singular points of the four singular cubics in the Hesse pencil. Each of these singular cubics   
   has three singular points. Thus each choice of three of the four
   singular cubics in the Hesse pencil defines 9 points, and these 9 points
   are the base points of a cubic pencil. The four pencils defined this way are
   precisely the four pencils obtained from the eight Halphen cubics.)
   The specific members chosen from each Hesse singular point cubic pencil 
   depend on the $3m$-torsion points of $T$, so we refer to the
   specific cubics chosen as the \emph{Halphen cubics
   of order $m$}.

   Our construction relies on classical facts known for the so-called
   \emph{dual Hesse configuration} of 9 lines in the plane meeting
   by threes on 12 points, and on the geometry of the curve $T$.
   The Weierstrass function $\wp'$ is a morphism to $\P^1$
   of degre 3, which in the case of $T$ is triply ramified at 3 points.
   For each positive integer $m$, we denote by $\P^1[n]$ the images
   of the $n$-torsion points of $T$.
   As noted above, the Hesse singular point cubic pencils are defined by taking as base points subsets
   of 9 points among the 12 vertices of the dual Hesse configuration;
   each pencil has 3 reducible members
   which are composed of lines of the configuration.
   We parameterize the pencils so that the reducible members
   correspond to parameters $u\in \P^1$ belonging to the branch locus of
   $\wp'$. Then the Halphen cubics of order $m$ are defined as the
   cubics in the Hesse singular point cubic pencils with parameters in $\P^1[3m]$.
   Our first result is the following:
\begin{introtheorem}\label{Halphen_sing}
For each $n\in 3\N$, let $H(n)$ be the union of all Halphen cubics
of order $n/3$.
The singularities of $H(n)$ are: 12 points of multiplicity
$n^2-3$ at the vertices of the dual Hesse configuration, with
$n^2/3-1$ triple points infinitely near to them, and $(n^2-3)(n^2/3-3)$
quadruple points.
\end{introtheorem}

   As a corollary the Harbourne index of $H(n)$ tends to -4 as $n$ grows.
   The configuration of Halphen cubics therefore behaves like the
   Roulleau-Urz\'ua configuration over which it is modelled.
   Our second goal is to understand the rational map
   $T\times T \dashrightarrow \P^2$ used by Roulleau and Urz\'ua
   to construct their configuration. In Theorem \ref{LinSeriesThm}
   we determine the linear series associated to the map,
   and using the action of the theta group, we prove that
   both configurations agree:

\begin{introtheorem}\label{H=RU}
   The curves that form the Roulleau-Urz\'ua configuration corresponding
   to the $n$-torsion points for $n=3m$ are the Halphen cubics of $m$-th order.
\end{introtheorem}

The paper is organized as follows. In section \ref{sec:hesse-halphen}
we recall the classical construction of Hesse line configurations
and Halphen cubics, and we prove Theorem \ref{Halphen_sing}
along with additional information on the position of the singularities
(Theorem \ref{RU_on_plane}). Section \ref{sec:linseries}
is devoted to the study of the Roulleau-Urz\'ua configuration and
the map $T\times T \dashrightarrow \P^2$, and culminates in the
proof of Theorem \ref{H=RU}.

%*****************************************************************************

\section{The Hesse configurations and Halphen cubics}
\label{sec:hesse-halphen}

Recall the construction of the so-called Hesse configurations
(a modern account of this classical subject can be found in the
book \cite[Section 3.1]{Dolgachev}, see also
\cite{ArtebaniDolgachev09} and references therein for
their history and attributions).

Given a smooth plane cubic $C$ and a line $\ell$ joining two
flexes of $C$, the third intersection of $\ell$ and $C$ is
another flex. Altogether, there are  12 such \emph{lines of flexes},
and at each of the 9 flexes of $C$ exactly 4 of the 12 lines meet.
This configuration of lines and points is classically called
the \emph{Hesse line arrangement}
$(12_3,9_4)$; it does not depend on the choice of a cubic, i.e.,
the sets of 9 flex points of any smooth plane cubic are projectively
equivalent.

A triangle containing all 9 flexes is called a
\emph{triangle of flexes};
there are 4 such triangles
which together form the Hesse arrangement.
They can be obtained as follows.
Fix one of the flexes $p_0$ as the zero for the group law on $C$;
then the 9 flexes of $C$ form the 3-torsion subgroup $C[3]\cong(\Z/3\Z)^2$.
Each line of flexes $\ell$ through $p_0$ cuts on $C$ one of
its 4 cyclic subgroups of order 3.
The two cosets of this subgroup in $C[3]$ correspond to the two
lines of flexes which do not meet $\ell$ on $C$,
which toghether with $\ell$ form one of the triangles.

The polar curve of $C$ with respect to a flex $p$ is the degenerate conic
consisting of the tangent $T_pC$  and another line, called
the \emph{harmonic polar} of $p$. If $\ell$ is a line of flexes,
the harmonic polars of the three flexes on $\ell$ are concurrent,
and their point of intersection is the vertex opposite to $\ell$ in
the triangle of flexes to which it belongs.
Thus, each of the 9 harmonic polars goes through 4
vertices, one on each triangle of flexes, and at each of
the 12 vertices exactly 3 of the 9 lines meet. They
form the \emph{dual Hesse arrangement}
$(9_4,12_3)$, which again does not depend on the cubic.

The given curve $C$ and each triangle of flexes are cubics
through the 9 flex points. It follows that the 9 points of the Hesse configuration
are the base points of a pencil of cubics, called the \emph{Hesse pencil}; all
cubics in the pencil have the same flex points, and
every plane cubic is projectively
equivalent to one of the curves in the pencil. The Hesse pencil has
4 singular members, namely the 4 triangles of flexes.
We denote them $T_v, T_h, T_\delta, T_\gamma$.
For convenience we also fix the following notations
for the whole paper: $\T=\{v,h,\delta,\gamma\}$
will be the set of indices for the triangles;
$V_v=\{v_0, v_1, v_\infty\}$ will be the vertices of $T_v$; and similarly
for $V_h=\{h_0, h_1, h_\infty\}$, $V_\delta=\{\delta_0, \delta_1, \delta_\infty\}$
and $V_\gamma=\{\gamma_0, \gamma_1, \gamma_\infty\}$ (the symbols
$v, h, \delta, \gamma,$ and $0,1,\infty,$ are chosen to match
with constructions to appear later on). Moreover
we take $V=\bigcup_{t\in \T}V_t$ to denote the whole set of vertices
of the dual Hesse configuration, and for each $t\in \T$,
$\Lambda_t=V\setminus V_t$ to be the complement in $V$ of the set of vertices of
$T_t$.

The group of projective transformations of the plane which preserve the Hesse
arrangement (or equivalently the Hesse pencil) is a finite group $G_{216}$
called the \emph{Hesse group}. It obviously acts on the set
$\{T_v,T_h,T_\delta,T_\gamma\}$, and the image of the corresponding
representation $G_{216}\rightarrow S_4$ is the alternating
group $A_4$. Since it is not the full $S_4$ group, the order of the
indices $v, h, \delta, \gamma$ is not
entirely innocuous; we now introduce coordinates in order to be precise,
and for later use in the explicit determination of the Halphen cubics.

Take $C$ to be the Fermat cubic $x^3+y^3+z^3$, and let $\eps$
denote a  primitive third root of unity.
The 9 flex points of $C$
are the points $(-1,1,0)$, $(-1,\eps,0)$ and $(-1,\eps^2,0)$
and the 6 other points obtained from these by permutation;
it is customary to take as generators for the Hesse pencil
the Fermat cubic $C$ and the triangle $T_v=xyz$.
Thus the three coordinate points
      \be
         && v_0=(1:0:0),\quad v_1=(0:1:0),\quad v_\infty= (0:0:1)
      \ee
are three of the vertices of the dual Hesse configuration; the
remaining points are
      \be
         && h_0=(1:1:1),\quad h_1=(1:\eps:\eps^2),\quad h_\infty=(1:\eps^2:\eps) \\
         && \delta_0=(\eps:1:1),\quad \delta_1=(1:\eps:1),\quad \delta_\infty=(1:1:\eps) \\
         && \gamma_0=(\eps^2:1:1),\quad \gamma_1=(1:\eps^2:1),\quad \gamma_\infty= (1:1:\eps^2)
      \ee
The equations of the harmonic polar lines forming the dual Hesse configuration
   can be obtained using the coordinates of the points.
   For example the line through the
   points $v_0$ and $h_0$ has equation $z-y$ and also goes through
   $\delta_0$ and $\gamma_0$. We denote by $L_{i,j,k,l}$
   the line through $v_{u_i}$, $h_{u_j}$, $\delta_{u_k}$ and
   $\gamma_{u_l}$,
   so the line just considered is $L_{0000}$.
   Evaluating all collinearities we obtain
   \begin{equation}\label{eqn:lines}
   \begin{array}{lll}
      L_{0000}= z-y,\quad & L_{01\infty1}= z-\eps y,\quad & L_{0\infty1\infty}= \eps z-y\\
      L_{1011}= z-x,\quad & L_{110\infty}= \eps z- x,\quad & L_{1\infty\infty1}= z-\eps x\\
      L_{\infty0\infty\infty}= y-x,\quad & L_{\infty110}= y-\eps x,\quad & L_{\infty\infty01}= \eps y-x
   \end{array}
   \end{equation}

 For each $t \in \T$, the linear system of cubics
 through the 9 points in $\Lambda_t$ is a pencil $\CC_t$, which we call
 a Hesse singular point cubic pencil. It has three
 singular members, namely, for each vertex $t_u$ of $T_t$, $u\in \{0,1,\infty\}$,
 the union of the three harmonic polars concurrent at $t_u$ is a member
 of the pencil, which we call $C_{t_u}$.
 All nonsingular members of each pencil have $j$-invariant
 equal to zero. Halphen showed that the locus of
$9$-torsion points of all curves in the Hesse pencil consists of two members of
these pencils  $\CC_h, \CC_v, \CC_\delta, \CC_\gamma$
(so-called Halphen cubics);
the configurations of plane cubics described by Roulleau and Urz\'ua
consist of members of the same pencils (which we call
\emph{higher order Halphen cubics}).

Later on we shall give explicit equations of the Halphen cubics; for this purpose
we fix a coordinate $u$ in
each pencil $\CC_t$, such that $\{0,1,\infty\}$ correspond to the singular
members. Since the harmonic polars concurrent at $t_u$,
 in the notation above, are the
 $L_{u_v u_h u_\delta u_\gamma}$ with $u_t=u$,
the three singular members of $\CC_v$ are
\[
\begin{aligned}
      C_{v_0}\;\,&=\eps^2\prod L_{0***}\;\,
         =\eps^2 (z-y)(z-\eps y)(\eps z-y)   =  z^3-y^3 \\
      C_{v_1}\;\,&=\eps^2 \prod L_{1***}\;\,
         =\eps^2 (z-x)(\eps z- x)(z-\eps x)  =  z^3-x^3 \\
      C_{v_\infty}&=\eps^2 \prod L_{\infty***}
         =\eps^2 (y-x)(y-\eps x)(\eps y-x)   = y^3-x^3
\end{aligned}
\]
where the factor $\eps^2$ serves only a simplification purpose.
Similarly, the three singular memebers of $\CC_h$ are
\[
\begin{aligned}
      C_{h_0}\;\,=&&\prod L_{*0**}\;\,
	&=(z-y)(z-x)(y-x) = \\&&&-x^2y+xy^2+x^2z-y^2z-xz^2+yz^2 \\
      C_{h_1}\;\,=&&-\prod L_{*1**}\;\,
	&=-(z-\eps y)(\eps z- x)(y-\eps x)  = \\&&& \eps^2(x^2y-\eps^2xy^2-\eps^2x^2z+y^2z+xz^2-\eps^2yz^2) \\
      C_{h_\infty}=&&\prod L_{*\infty**}&=
      (\eps z-y)(z-\eps x)(\eps y-x)  = \\&&& -\eps(x^2y-\eps xy^2-\eps x^2z+y^2z+xz^2-\eps yz^2)
\end{aligned}
\]
   where the signs   are chosen so that $C_{h_0}+C_{h_\infty}=C_{h_1}$.
Finally, the singular members of $\CC_\delta$ and $\CC_\gamma$ are
\[
\begin{aligned}
       C_{\delta_0}\;\,=&&  \prod L_{**0*}\;\,
       &=(z-y)(\eps z-x)(\eps y-x) = \\&&&
       -x^2y+\eps xy^2+x^2z-\eps^2y^2z-\eps xz^2+\eps^2yz^2 \\
       C_{\delta_1}\;\,=&&  -\prod L_{**1*}\;\,
      &=-(z-\eps y)(z-\eps x)(y-x)   = \\&&&
      \eps^2(x^2y-xy^2-\eps^2x^2z+\eps^2y^2z+\eps xz^2-\eps yz^2) \\
      C_{\delta_\infty}=&&\prod L_{**\infty*}
      &=(\eps z-y)(z- x)(y-\eps x) =\\&&&
       -\eps(x^2y-\eps^2xy^2-\eps x^2z+\eps^2y^2z+\eps xz^2-yz^2)
\end{aligned}
\]
\[
\begin{aligned}
      C_{\gamma,0}\;\,=&&  \prod L_{***0}\;\,
       &=(\eps z-y)(\eps z-x)(y-x) = \\&&&
       -x^2y+xy^2+\eps x^2z-\eps y^2z-\eps^2xz^2+\eps^2yz^2 \\
      C_{\gamma,1}\;\,=&&-\prod L_{***1}\;\,
       &=-(z-\eps x)(y-\eps x)(z-y) = \\&&&
       \eps^2(x^2y-\eps^2xy^2-x^2z+\eps y^2z+\eps^2xz^2-\eps yz^2) \\
      C_{\gamma,\infty}=&&\prod L_{***\infty}
      &=(z-\eps y)(z-x)(\eps y-x) = \\&&&
      -\eps(x^2y-\eps xy^2-\eps^2x^2z+\eps y^2z+\eps^2xz^2-yz^2)
\end{aligned}
\]
With these choices, the member of the pencil $\CC_t$ corresponding
to the parameter $u\in \C$ will be
$C_{t_u}=C_{t_0} + u\,C_{t_\infty}$ for every $t\in \T$.

\paragraph{Halphen's Theorem.}
Fixing any of the flex points of a plane cubic curve $C$
as the zero point in the group law,
the set of $n$-torsion points for $n=3m$ coincides with the set of points
$p$ such that there exists a (possibly reducible) curve of degree $m$
meeting $C$ only at $p$;
%(this fact seems to have been first mentioned by Halphen).
thus the set of $3m$-torsion points of a cubic
is well defined and independent of the choice of a flex.
Halphen \cite{Halphen:recherches} studied the locus of points
of $3m$-torsion of all cubics in the Hesse pencil. For $m=1$
it consists on the 9 base points, as already said. For $m=2$,
it is the union of the 9 harmonic polars. For $m=3$ it is the union of
8 cubics, two in each pencil $\CC_u$, which we shall
describe next. For $m>3$, it is the union of 8 or 9 curves of higher
degrees, depending on the divisibility of $m$ by 3.

\begin{lemma}\label{triplecover}
   Let $T$ be an elliptic curve and $f:T\rightarrow \P^1$ a morphism of
   degree 3
   with triple ramification at 3 points. Then the $j$-invariant of $T$ is 0,
   and $f$ is unique up to translation and inversion on $T$, and
   up to automorphisms of $\P^1$.
\end{lemma}

\begin{proof}
   By the Riemann-Hurwitz formula, there is no more ramification than
   the three triple points, so
   $E$ is a 3 sheeted cover of $\P^1$ away from these points.
   Permuting the sheets of the cover induces an automorphism of $E$
   of order 3.
   But the only elliptic curves with an order 3 automorphism are those of $j$-invariant 0,
   and these support a single order 3 automorphism (up to translation and inversion),
   namely complex multiplication by a cube root of 1 \cite[Theorem 10.1]{Silverman}.
\end{proof}

The automorphisms of $\P^1$ which appear in lemma \ref{triplecover}
must obviously preserve the branch locus of $f$. To take care of these we
call \emph{marked line} a pair $\P^1_{M}=(\P^1, M)$ where $M=\{p,q,r\}$ is a set
of 3 distinct points in $\P^1$. A morphism of marked lines is an algebraic morphism
which preserves the markings (and so any two marked lines are isomorphic).
The group of automorphisms of the marked line is
finite, isomorphic to $S_3$, and by choosing a coordinate $u$ on $\P^1$
such that the marking is $M=\{0,1,\infty\}$, it is generated by the
morphisms $u\mapsto u^{-1}$ and $u\mapsto 1-u$.
Then, given an equianharmonic elliptic curve $T$ and a marked line $\P^1_{M}$,
there exixts a morphism  $f:T\rightarrow \P^1$ of
degree 3, with triple ramification over $M$, unique up to translation
and inversion on $T$, and up to the action of $S_3$ on $\P^1_{M}$.

In classical terminology, elliptic curves with $j$-invariant equal
to zero are called \emph{equianharmonic}.
In the setting of lemma \ref{triplecover}, consider the grup law on $T$ with
the zero at one of the ramification points.
Such a curve $T$ is obtained as
$T=\C/\Z[\zeta]$, where $\zeta=e^{2\pi i/6}$;
a degree 3 map is given by the Weierstrass derivative function
$\wp':T\rightarrow \P^1$;
and the corresponding order three automorphism
is induced by multiplication by $\eps=\zeta^2$. Its $3$ fixed points, which are
the ramification points of $\wp'$, are
$p_k=\frac k3 +\frac k3\zeta$ for $0\le k\le 2$.
Their images form the branch locus $M=\br(f)=\{i,-i,\infty\}\subset \P^1$.

For every choice of the zero at one of the ramification points
of $f$, the other two ramification points
are of $3$-torsion (in fact they form a cyclic subgroup of order 3,
as can be seen in the description via the $\wp'$ function)
hence for every positive integer $m$
the set $T[3m]$ of $3m$-torsion points on $T$
independent on which ramification
point is chosen as zero.
\begin{remark}\label{parameters}
Since the permuting of the 3 sheets is an automorphism,
if a (non-ramification) point in $T$ is of $3m$-torsion, then the
remaining two points in its fiber are $3m$-torsion as well.
So by lemma \ref{triplecover} the set of images
\[ \P^1_M[3m]=f(T[3m]) \setminus M\]
(where $M=\br(f)$ is the branch locus) is a well defined set
(independent of any choices and invariant under isomorphisms
of marked lines) consisting of $((3m)^2-3)/3=3m^2-1$ distinct points.
We call $\P^1_M[3m]$
the set of \emph{equianharmonic $3m$-torsion parameters}
relative to the markings $M$, or simply the set of equianharmonic $3m$-torsion parameters (denoted $\P^1[3m]$) if $M=\{0,1,\infty\}$.
\end{remark}

\begin{corollary}\label{triplecover-parameters}
    Let $T$ be an elliptic curve and $f:T\rightarrow \P^1$ a morphism of
   degree 3 with triple ramification at 3 points, and let $M=\br(f)$.
   The set $f^{-1}(M\cup \P^1_M[3m])$ is a translate of the subgroup of
   $3m$-torsion.
   \end{corollary}

We will be interested in the explicit determination of $\P^1[3m]$.
The first case corresponds to 3-torsion points, $m=1$,
for which there are $3\cdot 1^2-1=2$ equianharmonic torsion parameters.
Invariance under the action of $S_3$ is enough to determine $\P^1_M[3]$ as
the \emph{Hessian pair} of the marking $M$,
which can also be characterized as the two points which together with
$M$ form an \emph{equianharmonic set} (i.e., with cross ratio a cube root of 1).
For $M=\{0,1,\infty\}$ one gets
$\P^1[3]=\{-\eps,-\eps^2\}$ (this is the only set of 2 points invariant for both
$t\mapsto t^{-1}$ and $t\mapsto 1-t$).
 $\P^1[6]$, although more involved, can be computed using
the $S_3$-action as well, but this is not the case for higher $m$.
Below we compute $\P^1[6]$ from the definition, a method that does generalise to all
$m$.

\begin{theorem}[Halphen, {\cite[\S 3]{Halphen:recherches}}]
Consider each pencil $\CC_t$, $t\in\T$ as a marked line,
where the marking consists of the three singular members $C_{t_u}$
with $u\in\{0,1,\infty\}$. Denote $\CC_t[3]$ the set of
two cubics in the pencil that correspond to equianharmonic $3$-torsion parameters.
Then
\[ \CC_h[3]\cup\CC_v[3]\cup\CC_\delta[3]\cup\CC_\gamma[3]\]
is the locus of $9$-torsion points of curves in the Hesse pencil.
\end{theorem}

   Using the fact that $\P^1[3]=\{-\eps,-\eps^2\}$,
   the 8 Halphen cubics can be explicitly written as follows.
%   \begin{center}
%   \begin{tabular}{lll}\tline
%   Pencil & Cubic 1 & Cubic 2 \\ \tline
%      $\CC_v$ & $C_{v_{-\eps}}=x^3+\eps^2y^3+\eps z^3$ &
%      	$C_{v_{-\eps^2}}=x^3+\eps y^3+\eps^2z^3$ \\
%      $\CC_h$ & $C_{h_{-\eps}}=x^2y+y^2z+xz^2$ &
%      	$C_{h_{-\eps^2}}=xy^2+x^2z+yz^2$ \\
%      $\CC_{\delta}$ & $C_{\delta_{-\eps}}=x^2y+\eps^2y^2z+\eps xz^2$ &
%      	$C_{\delta_{-\eps^2}}=xy^2+\eps^2x^2z+\eps yz^2$ \\
%      $\CC_{\gamma}$ & $C_{\gamma_{-\eps}}=x^2y+\eps y^2z+\eps^2xz^2$ &
%      	$C_{\gamma_{-\eps^2}}=xy^2+\eps x^2z+\eps^2yz^2$ \\ \tline
%   \end{tabular}
%   \end{center}
\[
\begin{aligned}
      C_{v_{-\eps}}&=x^3+\eps^2y^3+\eps z^3 &
      	C_{v_{-\eps^2}}&=x^3+\eps y^3+\eps^2z^3 \\
      C_{h_{-\eps}}&=x^2y+y^2z+xz^2 &
      	C_{h_{-\eps^2}}&=xy^2+x^2z+yz^2 \\
      C_{\delta_{-\eps}}&=x^2y+\eps^2y^2z+\eps xz^2 &
      	C_{\delta_{-\eps^2}}&=xy^2+\eps^2x^2z+\eps yz^2 \\
      C_{\gamma_{-\eps}}&=x^2y+\eps y^2z+\eps^2xz^2 &
      	C_{\gamma_{-\eps^2}}&=xy^2+\eps x^2z+\eps^2yz^2
\end{aligned}
\]
This same list also arises in a somewhat different way in \cite{ADPHR93, Fr02}
and \cite[Proposition 5.2]{ArtebaniDolgachev09}, where a modern account of Halphen's theorem is given.
In addition to computing the 9-torsion of members of the Hesse pencil,
the Halphen cubics are special in their Hesse singular point cubic pencils, in the
following ways. First, the base points of the pencil are a translate
(with respect to the group law of the Halphen cubic)
of its set of 3-torsion points; and second, they intersect one another
only in base points, tangently in sets of three (from distinct pencils).
It is these latter properties that we
seek to generalize in higher order Halphen cubics.
Denote $\CC_t[3m]$ the union of the $3m^2-1$ cubics in the pencil
$\CC_t$ corresponding to equianharmonic $3m$-torsion parameters, and call them
the \emph{Halphen cubics of order $m$}. In order to better describe their
intersections we consider the blow up $X\rightarrow \P^2$ at the 12
points of $V$, and denote $\tilde{\CC}_t[3m]$ the strict transforms.
Denote $L$ the pullback to $X$ of the class of a line, $E_{t_u}$
the exceptional divisor above the point $t_u$,
$E$ the sum of all 12 exceptionals, and
\[E_t= \sum_{t'_u\in \Lambda_t} E_{t'_u} =
\sum_{\substack{t'\ne t \\ u\in \{0,1,\infty\}}} E_{t'_u}\]
the divisor above the 9 point set $\Lambda_t$;
thus each Hesse singular point cubic pencil can be written
$\CC_t=|3H-E_t|$.
Each exceptional divisor $E_{t_u}$ carries a natural marking
\begin{equation}\label{marking}
 M_{t_u}=E_{t_u}\cap \tilde C_{t_u}
\end{equation}
consisting of the
directions of the three lines in the dual Hesse configuration going through
the point $t_u$ (these are the component lines of $C_{t_u}$).
In the sequel, unless explicitly specified, the equianharmonic torsion
points on $E_{t_u}$ will always be considered with respect to the
natural marking, and we denote them
$E_{t_u}[3m]=(E_{t_u})_{M_{t_u}}[3m]$.

Note that, for each $t'\ne t$ we have
\begin{equation}
  \label{eq:specialfibers}
E_{t_u}\cap\left(\tilde C_{t'_0} \cup \tilde C_{t'_1} \cup \tilde C_{t'_\infty} \right)
= E_{t_u}\cap\tilde C_{t_u}=M_{t_u}.
\end{equation}

\begin{theorem}\label{RU_on_plane}
The reducible curve $H_m=\tilde{\CC}_h[3m] \cup \tilde\CC_v[3m]
\cup \tilde\CC_\delta[3m]\cup\tilde\CC_\gamma[3m]$ belongs
to the linear system $|12(3m^2-1)L-9(3m^2-1)E|$, and has the following
singularities:
\begin{enumerate}
\item $3m^2-1$ ordinary triple points on each $E_{t_u}$, for a total of $12(3m^2-1)$
triple points;
\item $9(3m^2-1)(m^2-1)$ ordinary quadruple points off $E$.
\end{enumerate}
Moreover, each component of the curve passes through 9 of the triple
points and $9(m^2-1)$ of the quadruple points, which together constitute
a translate of its $3m$-torsion subgroup.
\end{theorem}
\begin{proof}
The linear equivalence class is clear from the fact that
\(
\CC_t[3m]\sim (3m^2-1)\left(
3L-E_t
\right).
\)

Each pencil $\tilde\CC_t$ induces an elliptic fibration
$\phi_t:X\rightarrow \P^1$ for which the 9 exceptional components above
$\Lambda_t$ are sections. For every exceptional component
$E_{t_u}$, three of the fibrations
have it as a section (those $\phi_{t'}$ with $t'\ne t$), and
for the remaining fibration $\phi_t$, it is a component of a fiber,
with multiplicity 3. Indeed,
\begin{equation}\label{eq:reducible_fibers}
 \phi_t^{-1}(\phi_t(C_{t_u}))=\tilde C_{t_u}+3E_{t_u}
\end{equation}
because $C_{t_u}$ has multiplicity 3 at $t_u$.

Any two pencils among the 4 share 6 base points, and hence
have intersection number $3^2-6=3$. The restriction of
$\phi_h,\phi_\delta$ and $\phi_\gamma$
%have intersection number 3,
%as for example $\CC_h=|3L-E_{v,0}-E_{v,1}-E_{v,\infty}-
%E_{\delta,0}-E_{\delta,1}-E_{\delta,\infty}-E_{\gamma,0}-E_{\gamma,1}-E_{\gamma,\infty}|$
%shares 6 bas
to any smooth curve $\tilde C_{v_s}$, $s\in \C\setminus\{0,1\}$
in the pencil $\tilde \CC_v$ gives therefore
a morphism of degree 3, $\phi_t|_{\tilde C_{v_s}}:\tilde C_{v,s}\rightarrow E_{v_0}$.
Since $t_u$  is a base point of $\CC_v$ for each $t\ne v$ and $u\in\{0,1,\infty\}$
$\tilde C_{v_s}$ meets $E_{t_u}$ at a point $p=p_{t_u}$ and
by \eqref{eq:reducible_fibers},
\[
 \left(\phi_t|_{\tilde C_{v_s}}\right)^{-1}(C_{t_u}\cap E_{v_0})=
 \tilde C_{h_s}\cap \left(\tilde C_{t_u} + 3 E_{t_u}\right) \ge 3p,
\]
but since the degree of the map is 3, the inequality must in fact
be an equality
(so $\tilde C_{h_s}$ meets no point on $\tilde C_{t_u}$ and
the intersection with $E_{t_u}$ is transversal).
This holds for all  $u\in\{0,1,\infty\}$, so
$\phi_t|_{\tilde C_{h_s}}$ has triple ramification
above $M_{h_0}$. As the action of the Hesse group on $\T$ is 2-transitive,
for every $t\ne t'$ the restriction of $\phi_t$
to a fiber $\tilde C_{t'_s}$ has triple ramification
above $M_{t'_0}$.
Since $\phi_t|_{C_{t'_s}}$ has no additional ramification (by Riemann-Hurwitz)
all intersections between $\tilde C_{t'_s}$ and nonsingular fibers
of $\phi_t$ are transversal.
Therefore, intersections between components of $H_m$
(which are nonsingular fibers of the pencils)
are transversal; this means that all multiple points of $H_m$,
which are generated by such intersections, are ordinary.
Moreover, by corollary \ref{triplecover-parameters},
\[\tilde C_{t'_s}\cap (\tilde\CC_{t}[3m]\cup E_{t_0}\cup E_{t_1}\cup E_{t_\infty})=
\phi_t|_{C_{t'_s}}^{-1} \left( E_{t'_0}[3m] \cup M_{t'_0} \right)\]
consists of the $3m$-torsion points of $C_{t'_s}$ up to translation, for each
$t'\ne t$.

By construction,  for each $t\ne t'$, the intersection
$\tilde\CC_t[3m]\cap E_{t'_0}$
consists of the equianharmonic $3m$-torsion parameters
$E_{t'_0}[3m]$, i.e., through each point of $E_{t'_0}[3m]$ there are three
components of $H_m$, one in each pencil $\CC_t$, $t\ne t'$.
These points are therefore triple points of $H_m$, and
there are $3m^2-1$ such points on each of the 12 exceptional components.

Taking into account the linear equivalence class of the $\CC_t$
and the intersection product on $X$, we
see that besides the triple points, each pair
$\tilde\CC_t[3m]$, $\tilde\CC_{t'}[3m]$ intersect at
$9(3m^2-1)(m^2-1)$ additional points. The proof will be complete by showing
that these belong to the two remaining $\tilde\CC_{t''}[3m]$'s.
We accomplish this by proving that, for every component $\tilde C_{v_s}$
of $\tilde\CC_v[3m]$, the sets
\be
&& A_h=\tilde C_{v_s}\cap(\tilde\CC_h[3m]\cup
E_{h_0}\cup E_{h_1} \cup E_{h_\infty})\\
&& A_\delta=\tilde C_{v_s}\cap (\tilde\CC_{\delta}[3m]\cup
E_{\delta_0}\cup E_{\delta_1} \cup E_{\delta_\infty})
\ee
are equal; then
\[\tilde\CC_v[3m]\cap\tilde\CC_h[3m] \setminus E=\tilde\CC_v[3m]\cap \tilde\CC_{\delta}[3m] \setminus E\]
and by symmetry all pairs  $\tilde\CC_t[3m]$, $\tilde\CC_{t'}[3m]$
intersect at the same set of $9(3m^2-1)(m^2-1)$ points,
which finishes the proof.

Indeed, by corollary \ref{triplecover-parameters} both $A_h$ and
$A_\delta$ are obtained from the set of $3m$-torsion
points of $\tilde C_{v_s}$ by suitable translations according
to the group law in $\tilde C_{v_s}$, therefore $A_h=t_p(A_\delta)$
for some point $p\in \tilde C_{v_s}$.
Since $\tilde C_{v_s}$ meets $E_{\gamma_0}$ at one of the triple points,
which must also belong to $\tilde\CC_h[3m]$ and
$\tilde\CC_{\delta}[3m]$, it follows that $A_h\cap A_\delta$ is nonempty.
Therefore $p$ is of $3m$-torsion, and $A_h=A_\delta$ as claimed.
\end{proof}

\paragraph{Explicit computation of the equianharmonic torsion parameters.}
   Our method to compute the higher order Halphen cubics geometrically is based
   on the following remark.

\begin{remark}
The plane cubic curve $C=x^3+y^3-z^3$ is a $j$-invariant 0 curve, and each
of the three lines $L_{1***}$ in the dual Hesse configuration
going through the point $v_1=(0,1,0)$ is a flex line for $C$
(i.e., tangent to $C$ at a flex point).
For brevity, in this section we denote
these lines simply
\be
      L_{1}=\, -\eps^2(z-x),\quad & L_{0}=\, \eps z- x,\quad &
      L_{\infty}=\, z- \eps x
\ee
(With respect to the equations $L_{1***}$ above, a product with
adequate constant coefficients was done so that $L_0+L_\infty=L_1$).
Thus the linear series on $C$ given by the pencil of lines through
$v_1$ defines a morphism $C\to E_{v_1}$ with triple ramification
above the points corresponding to the directions of the $L_{u}$,
which form exactly the set $M_{v_1}$;
and the equianharmonic torsion parameters (with respect to $M_{v_1}$)
can be computed as the projections to $E_{v_1}$ of the
torsion points on $C$.
\end{remark}

Once the $n$-torsion points on $C$ are known, their projections to
 $E_{v_1}$ (which means their $(x,z)$ coordinates)
 are the equianharmonic $n$-torsion parameters.
   The torsion points can be found in principle solving algebraic
   equations involving so-called division polynomials \cite{La78},
   so the method works uniformly for all $n$.
   In this section we will determine the equianharmonic 6-torsion parameters,
   where we can find the required 6-torsion points
   using a more geometric method.
   We then produce the 44 Halphen cubics of order 2.
   These include the 8 order 1 cubics, so we have 36 still to find.

\begin{lemma}
   The 6-torsion points of $C$ are obtained by adding
   (using the group law on $C$)
   each
   of the nine points
   \be
      (1,0,1), (1,0,\eps), (1,0, \eps^2),
      (1,-1,0), (1,-\eps,0), (1,-\eps^2,0),
      (0,1,1), (0,1,\eps), (0,1,\eps^2)
   \ee
   to each of the four points
   \be
      (1,0,1), (1,-b,-1), (1,-\eps b,-1), (1,-\eps^2b,-1)
   \ee
   where $b^3-2=0$.
\end{lemma}

\begin{proof}
   The 6-torsion points can be obtained by adding
   3-torsion points and 2-torsion points.
   The 3-torsion points are the flex points,
   which as noted above, are the nine points given.
   As for the 2-torsion points,
   note first that $(1,0,1)$ is a flex. Regarding it as the identity for the group law on $C$,
   the 2-torsion points on $C$ are the lines tangent to $C$ which go through the identity (i.e., through $(1,0,1)$).
   One can check that the required points are the four points given.
   Using the group law on $C$ one can now find all 36 of the 6-torsion points.
\end{proof}

\paragraph{Projections from 6-torsion points.}
   Three of the 6-torsion points are the ramification points
   of the projection,
   and 6 of them correspond to the $n=3$ case.
   The other 27 come from taking a line through a 2-torsion point and one of the nine
   3-torsion points. One then finds the line through each of these 27 points and the point $v_1$.
   This gives 9 lines through $v_1$:
   \be
      &&
      x-(1/2)b^2z,\quad
      x-(1/2)\eps b^2z,\quad
      x-(1/2)\eps^2b^2z
      \\
      &&
      x-bz,\quad
      x-\eps bz,\quad
      x-\eps^2bz
      \\
      &&
      x+z,\quad
      x+\eps z,\quad
      x+\eps^2z
   \ee
%   As an aside we note that three of them are defined over ${\mathbb Q}[\eps]$.
%
   The ramification points map to the lines $L_{i}$, and the lines
   above (up to product with a constant, in the same order) can be written as:
% change of basis:
% \eps 1
% -1 -\eps
   \[
   \begin{aligned}
      (b\eps-1)L_{0}&-\eps(b\eps^2-1)L_{\infty}, &
      (b-1)L_0&-\eps(b\eps-1)L_\infty, &
      (b\eps^2-1)L_0&-\eps(b-1)L_\infty \\
      (b\eps^2-1)L_0&-\eps^2(b\eps-1)L_\infty, &
      (b-1)L_0&-\eps^2(b\eps^2-1)L_\infty,&
      (b\eps-1)L_0&-\eps^2(b-1)L_\infty \\
      L_0&-L_\infty,&
      2 L_0&+ L_\infty,&
      L_0&+2L_\infty
   \end{aligned}
   \]
Now the equianharmonic torsion parameters $\P^1[6]\setminus\P^1[3]$
can be obtained as the ratios
between the coefficients of $L_{\infty}$ and $L_{0}$ in the previous
list. In the same order again:
   \[
   \begin{aligned}
\tau &=-\eps\frac{b\eps^2-1}{b\eps-1},&
(1-\tau)^{-1}&=-\eps\frac{b\eps-1}{b-1},&
1-\tau^{-1}&=-\eps\frac{b-1}{b\eps^2-1},\\
\tau^{-1}&= -\eps^2\frac{b\eps-1}{b\eps^2-1},&
(1-\tau^{-1})^{-1}&=-\eps^2\frac{b\eps^2-1}{b-1},&
1-\tau&=-\eps^2\frac{b-1}{b\eps-1},\\
&-1,&&\frac12,&&2
   \end{aligned}
   \]
As an aside we note that 3 of them are defined over ${\mathbb Q}$;
they form one orbit under the action of $S_3$
generated by $t\mapsto 1-t$ and $t\mapsto t^{-1}$.
The remaining 6 form another orbit, which consists of the roots
of the irreducible invariant polynomial
$P(x)=x^6-3x^5+5x^3-3x+1=(x^2-x-1)^3+2$, with the property that
$\Q[\tau]=\Q[\eps,b]$ is the splitting field of $x^3+2$ (see \cite[page 59]{Cohen}).

\paragraph{The 36 new Halphen cubics.}
   Now here are the 36 cubics we get, normalized to obtain a simple
expression (i.e., the polynomial $C_{v_\tau}$ as given on the list is actually a
scalar multiple of $C_{v_0}+\tau C_{v_\infty}$).

\[
\begin{aligned}
 C_{v_{\tau}}&= (b-\eps)x^3+\eps^2(b-1)y^3+(b\eps-1)z^3 \\
 C_{h_{\tau}}&=x^2y-\eps^2bxy^2-\eps^2bx^2z+y^2z+xz^2-\eps^2byz^2 \\
 C_{\delta_{\tau}}&=x^2y-bxy^2-\eps ^2bx^2z+\eps ^2y^2z+\eps xz^2-\eps byz^2 \\
 C_{\gamma_{\tau}}&=x^2y-\eps ^2bxy^2-bx^2z+\eps y^2z+\eps ^2xz^2-\eps byz^2
\end{aligned}
\]
%\vskip\baselineskip
\[
\begin{aligned}
 C_{v_{(1-\tau)^{-1}}}&= \eps(b\eps-1)x^3+\eps(b-\eps) y^3+(b-1)z^3\\
 C_{h_{(1-\tau)^{-1}}}&= x^2y-\eps bxy^2-\eps bx^2z+y^2z+xz^2-\eps byz^2 \\
 C_{\delta_{(1-\tau)^{-1}}}&=x^2y-\eps ^2bxy^2-\eps bx^2z+\eps ^2y^2z+\eps xz^2-byz^2 \\
 C_{\gamma_{(1-\tau)^{-1}}}&=x^2y-\eps bxy^2-\eps ^2bx^2z+\eps y^2z+\eps ^2xz^2-byz^2\\
 \end{aligned}
\]
\vskip\baselineskip
\[
\begin{aligned}
 C_{v_{1-\tau^{-1}}}&= \eps(b-1) x^3 +(b-\eps^2) y^3+ (b\eps^2-1)z^3 \\
 C_{h_{1-\tau^{-1}}}&= x^2y-bxy^2-bx^2z+y^2z+xz^2-byz^2\\
 C_{\delta_{1-\tau^{-1}}}&=x^2y-\eps bxy^2-bx^2z+\eps ^2y^2z+\eps xz^2-\eps ^2byz^2 \\
 C_{\gamma_{1-\tau^{-1}}}&=x^2y-bxy^2-\eps bx^2z+\eps y^2z+\eps ^2xz^2-\eps ^2byz^2
\end{aligned}
\]
\vskip\baselineskip
\[
\begin{aligned}
 C_{v_{\tau^{-1}}}&= (b-\eps^2) x^3+\eps(b-1) y^3 + (b\eps^2 -1) z^3 \\
 C_{h_{\tau^{-1}}}&=2x^2y-\eps^2b^2xy^2-\eps^2b^2x^2z+2y^2z+2xz^2-\eps^2b^2yz^2\\
 C_{\delta_{\tau^{-1}}}&=2x^2y-b^2xy^2-\eps ^2b^2x^2z+2\eps ^2y^2z+2\eps xz^2-\eps b^2yz^2 \\
 C_{\gamma_{\tau^{-1}}}&=2x^2y-\eps ^2b^2xy^2-b^2x^2z+2\eps y^2z+2\eps ^2xz^2-\eps b^2yz^2 \\
\end{aligned}
\]
\vskip\baselineskip
\[
\begin{aligned}
 C_{v_{(1-\tau^{-1})^{-1}}}&=\eps(b-\eps)x^3+\eps(b\eps-1) y^3+(b-1)z^3\\
 C_{h_{(1-\tau^{-1})^{-1}}}&= 2x^2y-\eps b^2xy^2-\eps b^2x^2z+2y^2z+2xz^2-\eps b^2yz^2\\
 C_{\delta_{(1-\tau^{-1})^{-1}}}&=2x^2y-\eps ^2b^2xy^2-\eps b^2x^2z+2\eps ^2y^2z+2\eps xz^2-b^2yz^2 \\
 C_{\gamma_{(1-\tau^{-1})^{-1}}}&=2x^2y-\eps b^2xy^2-\eps ^2b^2x^2z+2\eps y^2z+2\eps ^2xz^2-b^2yz^2
\end{aligned}
\]
\vskip\baselineskip
\[
\begin{aligned}
 C_{v_{(1-\tau)^{-1}}}&= \eps^2(b-1) x^3+(b-\eps) y^3+(b\eps-1)z^3\\
 C_{h_{(1-\tau)^{-1}}}&= 2x^2y-b^2xy^2-b^2x^2z+2y^2z+2xz^2-b^2yz^2\\
 C_{\delta_{(1-\tau)^{-1}}}&=2x^2y-\eps b^2xy^2-b^2x^2z+2\eps ^2y^2z+2\eps xz^2-\eps ^2b^2yz^2 \\
 C_{\gamma_{(1-\tau)^{-1}}}&=2x^2y-b^2xy^2-\eps b^2x^2z+2\eps y^2z+2\eps ^2xz^2-\eps ^2b^2yz^2
\end{aligned}
\]
\vskip\baselineskip
\[
\begin{aligned}
  C_{v_{-1}}&=x^3-2 y^3+z^3 \\
 C_{h_{-1}}&=x^2y+\eps^2xy^2+\eps^2x^2z+y^2z+xz^2+\eps^2yz^2 \\
 C_{\delta_{-1}}&=x^2y+xy^2+\eps ^2x^2z+\eps ^2y^2z+\eps xz^2+\eps yz^2 \\
 C_{\gamma_{-1}}&=x^2y+\eps xy^2+x^2z+\eps y^2z+\eps^2 xz^2+\eps yz^2
\end{aligned}
\]
\vskip\baselineskip
\[
\begin{aligned}
  C_{v_{\frac{1}2}}&=x^3+y^3-2z^3 \\
 C_{h_{\frac{1}2}}&=x^2y+\eps xy^2+\eps x^2z+y^2z+xz^2+\eps yz^2 \\
 C_{\delta_{\frac{1}2}}&= x^2y+\eps ^2xy^2+\eps x^2z+\eps ^2y^2z+\eps xz^2+yz^2\\
 C_{\gamma_{\frac{1}2}}&= x^2y+\eps xy^2+\eps ^2x^2z+\eps y^2z+\eps ^2xz^2+yz^2
\end{aligned}
\]
\vskip\baselineskip
\[
\begin{aligned}
  C_{v_{2}}&= 2x^3-y^3-z^3\\
 C_{h_{2}}&=x^2z+xz^2+x^2y+z^2y+xy^2+zy^2 \\
 C_{\delta_{2}}&=x^2y+\eps xy^2+x^2z+\eps ^2y^2z+\eps xz^2+\eps ^2yz^2 \\
 C_{\gamma_{2}}&=x^2y+xy^2+\eps x^2z+\eps y^2z+\eps ^2xz^2+\eps ^2yz^2 \\
\end{aligned}
\]

As explained in the proof of Theorem \ref{RU_on_plane}, the singular
points of the configuration $H_m$ are exactly the $3m$-torsion points on
each of its components, translated by one (arbitrary) base point of the
pencil to which it belongs. In the case of $H_2$, these 6-torsion points
can be computed for each $C_{t_s}$ by the same method above; we leave
the details to the interested reader.

\section{The linear series of the Roulleau-Urz\'ua map}
\label{sec:linseries}
% The configuration of the Halphen cubics of order $m$ has, by virtue of theorem
% \ref{RU_on_plane}, the same degree and singularities as the Roulleau-Urz\'ua
% configuration.
% In this section we obtain additional information on the map from an abelian
% surface to $\P^2$ which was used by Roulleau-Urz\'ua to construct their
% configuration.
%  eventually show that both configurations are in fact the same,

Recall now the construction by Roulleau and Urz\'ua
of their planar configuration of cubics (which we will eventually
show to be equal to the Halphen cubics of order $m$).
   Let as before $T=\C/\Z[\zeta]$, where $\zeta=e^{2\pi i/6}$, be the equianharmonic
   elliptic curve, and let $p_i=\frac i3+\frac i3\zeta$ for $0\le i\le 2$
   be the  $3$ fixed points of multiplication by $\zeta^2$.
   Let $A$ be the abelian surface $A=T\times T$, and denote
   $\sigma:A\to A$ the induced automorphism
   defined by $(x,y)\mapsto (\zeta^2 x,\zeta^2 y)$, which has 9 fixed points,
   namely
    \be
      p_{ij}=(\tfrac i3+\tfrac i3\zeta,\ \tfrac j3+\tfrac j3\zeta)
      \qquad (0\le i,j\le 2),
   \ee
  so that $p_{ij}=(p_i,p_j)$.

   Divisors of particular importance on $A$ are
   $V=0\times T$, $H=T\times 0$,
   the diagonal $\Delta$ and the graph $\Gamma$
   of the complex multiplication by $\zeta$.
In fact, these curves span the N\'eron-Severi group of $A$.
%Since $A=T^2$ where $T$ is an elliptic curve with complex multiplication,
%we have $NS(A)=4$.
%\cite{Kat75}.
%\meta{TB: ref removed }
%Each of $V, H, \Delta, \Gamma$
%meets each of the others once (indeed, all four contain the point $p_{00}$),
%the intersections are transverse, and,
%being curves on an abelian variety, their self-intersections are 0.
%It is now easy to check that the intersection matrix for $V, H, \Delta, \Gamma$
%is invertible, hence these divisors span.)
   Translating $V,H,\Delta,\Gamma$ by the fixed points $p_{ij}$
   gives twelve curves: each of the fixed points is on four of the translates
   and each translate contains 3 of the fixed points, as suggested in Figure \ref{FigureOnA}.
   Of the diagonal lines, only $\Delta$ and $\Gamma$ can be shown properly as going through
   three of the points $p_{ij}$, but the line through points $p_{01}$ and $p_{10}$ also goes through
   point $p_{22}$, and in general if $(i_1,j_1)+(i_2,j_2)+(i_3,j_3)$ add up as vectors in $\mathbb Z_3^2$ to
   $(0,0)$, then the points $p_{i_1j_1}$, $p_{i_2j_2}$ and $p_{i_3j_3}$,  are collinear.
   Moreover, the points of intersection actually occur only at the points $p_{ij}$,
   contrary to how it might look in the drawing.

\begin{figure}[htbp]
\begin{center}
\begin{tikzpicture}[line cap=round,line join=round,>=triangle 45,x=2.0cm,y=2.0cm]
\clip(-0.5759937642812904,-0.6736429837486906) rectangle (2.587531552309566,2.545012919896644);
\draw [domain=-0.8759937642812904:3.087531552309566] plot(\x,{(-0.0-2.0*\x)/-2.0});
\draw [domain=-0.8759937642812904:3.087531552309566] plot(\x,{(--1.0--1.0*\x)/1.0});
\draw [domain=-0.8759937642812904:3.087531552309566] plot(\x,{(-1.0--1.0*\x)/1.0});
\begin{scriptsize}
\draw [color=white, fill=white] (0.5,0.5) circle (3pt);
\draw [color=white, fill=white] (0.5,1.5) circle (3pt);
\draw [color=white, fill=white] (1.5,0.5) circle (3pt);
\draw [color=white, fill=white] (1.5,1.5) circle (3pt);
\end{scriptsize}
\draw (0.0,-0.6736429837486906) -- (0.0,2.545012919896644);
\draw (1.0,-0.6736429837486906) -- (1.0,2.545012919896644);
\draw (2.0,-0.6736429837486906) -- (2.0,2.545012919896644);
\draw [domain=-0.8759937642812904:3.087531552309566] plot(\x,{(-0.0-0.0*\x)/2.0});
\draw [domain=-0.8759937642812904:3.087531552309566] plot(\x,{(--2.0-0.0*\x)/2.0});
\draw [domain=-0.8759937642812904:3.087531552309566] plot(\x,{(--4.0-0.0*\x)/2.0});
\draw [domain=-0.8759937642812904:3.087531552309566] plot(\x,{(--4.0-2.0*\x)/2.0});
\draw [domain=-0.8759937642812904:3.087531552309566] plot(\x,{(--3.0-1.0*\x)/1.0});
\draw [domain=-0.8759937642812904:3.087531552309566] plot(\x,{(--1.0-1.0*\x)/1.0});
\draw (0.01,0.0) node[anchor=north west] {$p_{22}$};
\draw (0.09,1.22) node[anchor=north west] {$p_{20}$};
\draw (0.01,2.247979157323769) node[anchor=north west] {$p_{21}$};
\draw (1.12,.05) node[anchor=north west] {$p_{02}$};
\draw (1.12,1.05) node[anchor=north west] {$p_{00}$};
\draw (1.11,2.22) node[anchor=north west] {$p_{01}$};
\draw (2.12,.05) node[anchor=north west] {$p_{12}$};
\draw (2.12,1.05) node[anchor=north west] {$p_{10}$};
\draw (2.01,2.0) node[anchor=north west] {$p_{11}$};
\draw (.95,0.6) node[anchor=north west] {$V$};
\draw (0.35,1) node[anchor=north west] {$H$};
\draw (1.2,1.35) node[anchor=north west] {$\Delta$};
\draw (0.65,1.52) node[anchor=north west] {$\Gamma$};
\begin{scriptsize}
\draw [fill=black] (0.0,0.0) circle (1.5pt);
\draw [fill=black] (1.0,0.0) circle (1.5pt);
\draw [fill=black] (2.0,0.0) circle (1.5pt);
\draw [fill=black] (0.0,1.0) circle (1.5pt);
\draw [fill=black] (1.0,1.0) circle (1.5pt);
\draw [fill=black] (2.0,1.0) circle (1.5pt);
\draw [fill=black] (0.0,2.0) circle (1.5pt);
\draw [fill=black] (1.0,2.0) circle (1.5pt);
\draw [fill=black] (2.0,2.0) circle (1.5pt);
\end{scriptsize}
\end{tikzpicture}
\caption{The curves $V, H, \Delta, \Gamma\subset A$ and their translates by the fixed points $p_{ij}$}
\label{FigureOnA}
\end{center}
\end{figure}
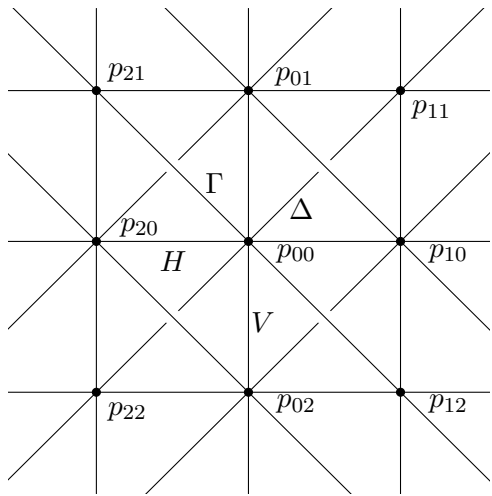

  Now let $B\to A$ be the blow up at the nine points $p_{ij}$.
  Since these are fixed points for $\sigma$, the automorphism lifts to $B$.
  By a slight abuse of notation, we denote
  the lift of $\sigma$ to $B$ again by $\sigma$ (and
  we denote the total transforms of $V,H,\Delta,\Gamma$ to $B$ with the same letters).
  Thus we have the quotient $B\to X=B/\generated{\sigma}$.
  Moreover, because $\sigma$ acts diagonally on $A$, $\sigma$ fixes tangent directions
  at $p_{00}$ (and hence also at each fixed point $p_{ij}$), so the fixed points for $\sigma$ acting on $B$
  are exactly the points of the exceptional curves for the nine points $p_{ij}$. Thus the
  ramification locus for the quotient map $B\to X$ is the union of these nine exceptional curves.

   Roulleau and Urz\'ua show in \cite{RU:Chern}
   that $X$ is smooth and rational, and that under the quotient $B\to X$
   the images of the 12 curves obtained from $V,H,\Delta,\Gamma$
   by translation are disjoint $(-1)$-curves
   whose contraction gives a birational morphism $X\to \P^2$,
   representing $X$ as
   the blow-up of $\P^2$ at the twelve points of the dual Hesse configuration.
   %
   % dual to the 12 lines
   % through pairs of flex points of a smooth plane cubic.
   % (Alternatively, we recall that the 9 flex points of any smooth plane cubic
   % are projectively equivalent. Altogether, there are only 12 lines that contain
   % two of these 9 points. These 9 points and the 12 lines through pairs of the points
   % are known as the Hesse arrangement. The 9 lines dual
   % to these 9 points give a singular curve of degree 9 having exactly 12 singular points.
   % These 12 points are the duals of the 12 lines of the Hesse configuration.
   % At each of these 12 dual points exactly 3 of the 9 dual lines meet, so they are triple
   % points of the arrangement of 9 dual lines.
   % Blowing up these 12 triple points gives the morphism $X\to \P^2$, and
   % these 12 points are the images under the composition $B\to X\to\P^2$
   % of the 12 translates of $V,H,\Delta,\Gamma$.)
   %
   So we
   have a diagram
   \be
      \renewcommand\arraystretch{1.5}
      \begin{array}{ccc}
         B=\Bl_{9}(A)     & \longrightarrow & A=T\times T \\
         \big\downarrow   &
         	\searrow\!\!\!\raise7pt\hbox{$\scriptstyle\varphi$}            & \\
         X=\Bl_{12}(\P^2) & \longrightarrow & \P^2 \hfill
      \end{array}
   \ee
   where the
   vertical map is of degree 3
   and its branch locus
   is the union of the
   nine exceptional curves for the upper horizontal map, whose images
   are the 9 harmonic polar lines of the dual Hesse configuration.

   In this section we describe the induced map
   $\varphi:B\to\P^2$
   in terms of linear series. Then, using the action of the theta
   group, we determine the coordinates of the images of the 12
   translates of $V,H,\Delta,\Gamma$, which will justify
   the choice of indices $v, h, \delta, \gamma$ in the previous section.

\begin{theorem}\label{LinSeriesThm}\
   \begin{itemize}
   \item[\rm a)]
      The morphism $\varphi:B\to\P^2$ is the map $\varphi_L$
      defined by the
      complete linear series $|L|$ associated with the line bundle
      \be
         L=V+H+\Delta+\Gamma-E
      \ee
      where $E$ is the sum of the nine exceptional divisors $E_{ij}$
      of the blow-up $B\to A$.
   \item[\rm b)]
%      Write $p_i=\frac i3+\frac i3\zeta$ with $0\le 1\le 2$
%      for the three fixed points of $\zeta^2$ on $T$,
      Consider the translates $V_i=V+(p_i,0)=V+p_{i0}$.
      Then there are elliptic curves
      $N_{01}, N_{02}, N_{12} \subset A$ such that the divisors
      \be
         D_0 := V_1+V_2+N_{12} \\
         D_1 := V_0+V_2+N_{02} \\
         D_2 := V_0+V_1+N_{01}
      \ee
      belong to the linear series $|L|$.
      If we define the map $\varphi_L:B\to\P^2$
      by suitably scaled sections corresponding to the
      divisors $D_0,D_1,D_2$, then the images of the 12 translates
      of $V,H,\Delta,\Gamma$ are the 12 Hesse dual points
      $v_u, h_u,\delta_u,\gamma_u$.
%      \be
%         && (1:0:0),\quad (0:1:0),\quad (0:0:1) \\
%         && (1:1:1),\quad (1:\eps:\eps^2),\quad (1:\eps^2:\eps) \\
%         && (\eps:1:1),\quad (1:\eps:1),\quad (1:1:\eps) \\
%         && (\eps^2:1:1),\quad (1:\eps^2:1),\quad (1:1:\eps^2)
%      \ee
%      where $\eps$ is a primitive third root of unity.
   \end{itemize}
\end{theorem}
% \begin{remark}
% We note if we take the Fermat cubic $x^3+y^3+z^3$, its 9 Hesse points
% (i.e., its flex points) are the points $(-1,1,0)$, $(-1,\eps,0)$ and $(-1,\eps^2,0)$
% and the 6 other points obtained from these by permutation, while the duals of the
% 12 Hesse lines are the 12 points given in Theorem \ref{LinSeriesThm}(b),
% which are also the points of intersection of pairs of the lines
% dual to the 9 Hesse points.
% \end{remark}

   The proof of Theorem \ref{LinSeriesThm} is split into several intermediate
   steps filling the rest of the present section.
   We start by showing:

\begin{proposition}
   The morphism $\varphi:B\to\P^2$ is the map defined by the
   complete linear series
   \be
      |V+H+\Delta+\Gamma-E|
   \ee
   where $E$ is as above.
\end{proposition}

\begin{proof}
With respect to the blow up $B\to A$, the proper transforms of the curves
$V,H,\Delta,\Gamma$ are $V'=V-E_{02}-E_{00}-E_{01}$, $H'=H-E_{20}-E_{00}-E_{10}$,
$\Delta'=\Delta-E_{22}-E_{00}-E_{11}$ and $\Gamma'=\Gamma-E_{21}-E_{00}-E_{12}$,
where $E_{ij}$ is the exceptional curve for the blow up of $p_{ij}$.
These curves are mutually disjoint and meet $E_{00}$ transversely.
Since $V', H', \Delta'$ and $\Gamma'$ are preserved curvewise by $\sigma$
and $E_{00}$ is fixed pointwise, the images $V'', H'', \Delta''$ and $\Gamma''$ of $V', H', \Delta'$ and $\Gamma'$
under the quotient $B\to X$ are disjoint and meet the image $E_{00}'$ of $E_{00}$
transversely. Since $V'', H'', \Delta''$ and $\Gamma''$ are exceptional curves
which map to points under $X\to\P^2$, $E_{00}'$ maps to a smooth plane rational curve $C$, hence
of self-intersection $C^2=(E_{00}'+V''+H''+\Delta''+\Gamma'')^2 = (E_{00}')^2+4(2)+4(-1)$.
But $B\to X$ is a triple cover, so $(\varphi^*(E_{00}'))^2=3(E_{00}')^2$, and
has triple ramification along $E_{00}$, so $\varphi^*(E_{00}')=3E_{00}$. Thus
$-9=(3E_{00})^2=3(E_{00}')^2$, so $(E_{00}')^2=-3$ and $C^2=1$, hence $C$ is a line.

The pullback of $C$ to $B$ is $V'+H'+\Delta'+\Gamma'+3E_5'=V+H+\Delta+\Gamma-E$,
which we denote by $L$.
I.e., the map $\varphi$ is defined by a 3 dimensional
   subspace of $H^0(B,L)$, and the argument so far shows that
   $L=V+H+\Delta+\Gamma-E$.
   We will now prove that
   $h^0(B,L)=3$, which then implies that $\varphi$ is defined by
   the complete linear series $|L|$.

   First, we have $(V+H+\Delta+\Gamma)^2=12$, and therefore by
   Riemann-Roch on $A$ we get $h^0(A,V+H+\Delta+\Gamma)=6$. It
   is therefore enough to find three fixed points $q_1,q_2,q_3$ of $\sigma$
   that impose independent conditions on $|V+H+\Delta+\Gamma|$,
   i.e., such that there is a divisor in the linear series
   $|V+H+\Delta+\Gamma|$ passing through $q_1$ and $q_2$, but
   not through $q_3$ (this suffices since the divisor
   $V+H+\Delta+\Gamma$ is very ample
   by \cite[Theorem~2.3]{BS:On-tensor-products}).
   Consider to this end the point
   $x=(\frac13,0)$ on $A$.
   Lemma~\ref{lemma:matching-translation} implies that there is
   a point $z\in A$ such that the divisor
   $t^*_xV+H+t^*_z(\Delta+\Gamma)$ belongs to the linear series
   $|V+H+\Delta+\Gamma|$. Let $q_1$ and $q_2$ be any two of the
   three fixed points lying on $H$. Clearly none of the nine
   fixed points lies on $t^*_xV$, and there can be at most five
   of them on $t^*_z(\Delta+\Gamma)$. Therefore there exists a
   fixed point $q_3$ that lies neither on $H$ nor on
   $t^*_z(\Delta+\Gamma)$. The triple of points $q_1,q_2,q_3$
   thus satisfies the required condition.
\end{proof}

\begin{lemma}\label{lemma:matching-translation}
   For every pair of points $x,y\in A$ there exists a unique
   point $z\in A$ such that
   \be
      t^*_xV+t^*_yH+t^*_z(\Delta+\Gamma)
      \linequiv V+H+\Delta+\Gamma
   \ee
   The analogous statement holds for any permutation of the
   curves $V,H,\Delta,\Gamma$.
\end{lemma}

\begin{proof}
   Consider the homomorphism of groups
   \be
      \Phi:A\times A\times A & \to     & \Pic^0(A) \\
      (x,y,z)                & \mapsto & t^*_xV+t^*_yH+t^*_z(\Delta+\Gamma)-(V+H+\Delta+\Gamma)
   \ee
   For every pair $(x,y)\in A\times A$, the map
   $\Phi(x,y,\cdot)$ is a translate of the canonical
   homomorphism $A\to\Pic^0(A)$, $z\mapsto
   t^*_z(\Delta+\Gamma)-(\Delta+\Gamma)$, associated with the
   line bundle $\Delta+\Gamma$. Since this line bundle
   is of self-intersection 2, it
   gives a
   principal polarization and therefore
   its canonical homomorphism is in fact
   an isomorphism (see \cite[Prop.~2.4.9]{BL:CAV}) and thus the
   intersection $\ker\Phi\cap(\set{(x,y)}\times A)$ consists of
   exactly one point.
\end{proof}

\begin{proposition}
   The map $\psi$ that assigns
   to given points $x,y\in A$ the point $z$ as in the preceding
   lemma is given by
   \be
      \psi:A\times A      & \to     & A \\
      ((x_1,x_2),(y_1,y_2)) & \mapsto & \big(-2x_1-(1+\overline\zeta)y_2,\ -(1+\zeta)x_1-2y_2\big)
   \ee
\end{proposition}

\begin{remark}
   For the special case where $x$ and $y$ are among the
   nine fixed points of~$\sigma$,
%   \be
%      p_{ij}=(\tfrac i3+\tfrac i3\zeta,\ \tfrac j3+\tfrac j3\zeta)
%      \qquad (0\le i,j\le 2)
%   \ee
   we get with a calculation
   \be
      \psi(p_{ij}, p_{kl})=p_{il}
   \ee
   In other words, we have
   \be
      t^*_{p_{ij}}V+t^*_{p_{kl}}H+t^*_{p_{il}}(\Delta+\Gamma)
      \linequiv V+H+\Delta+\Gamma
   \ee
\end{remark}

\begin{proof}[Proof of the proposition]
   For a line bundle $M$ on $A$ denote as usual by $\phi_M$ the
   canonical homomorphism $A\to\Pic^0(A)$, $x\mapsto t^*_xM-M$.
   The point $z=\psi(x,y)$ is characterized by the condition
   $t^*_xV+t^*_yH+t^*_z(\Delta+\Gamma)\linequiv V+H+\Delta+\Gamma$,
   which is equivalent to
   $\phi_V(x)+\phi_H(y)+\phi_{\Delta+\Gamma}(z)=0$.
   This in turn implies that
   \begin{equation}\label{eqn:canonical-homs}
      \psi(x,y)=\phi_{\Delta+\Gamma}\inverse
         \Big(-\phi_V(x)-\phi_H(y)\Big)
   \end{equation}
   The issue therefore is to explicitly determine the
   canonical maps.
   As $V+H$ gives a principal polarization,
   $\phi_{V+H}$ is an isomorphism, and hence we can use its
   inverse to identify $\Pic^0(A)$ with $A$.
   In that sense, we will, by slight abuse of notation,
   denote the composed homomorphism
   $T\times T=A\tol^{\phi_M}\Pic^0(A)\tol^{\phi_{V+H}\inverse} A=T\times T$
   again by $\phi_M$.
   In this setup, $\phi_V$ and $\phi_H$ are given by the matrices
   \be
      \smallmatr{1 & 0 \\
            0 & 0}
      \quad\mbox{and}\quad
      \smallmatr{0 & 0 \\
            0 & 1}
   \ee
   respectively.
   We now determine the map $\phi_{\Delta+\Gamma}$ in these terms.
   Consider to this end
   the isomorphism $g:T\times T\to T\times T$, $(x,y)\mapsto (x,y-x)$.
   The analytic representation of $g$ and its dual map $\hat g$
   are
   \be
      \smallmatr{1  & 0 \\
            -1 & 1}
      \quad\mbox{and}\quad
      \smallmatr{1 & -1 \\
            0 & 1}
   \ee
   We have $g\inverse(H)=\Delta$, thus
   \be
      \phi_\Delta=\hat g\phi_H g
      =\smallmatr{1  & -1 \\
             -1 & 1}
   \ee
   We can proceed in the analogous way for $\phi_\Gamma$
   using the isomorphism $h:T\times T\to T\times T$,
   $(x,y)\mapsto(x,y-\zeta x)$.
   The analytic representations of $h$ and $\hat h$ are
   \be
      \smallmatr{1 & 0 \\
            -\zeta & 1}
      \quad\mbox{and}\quad
      \smallmatr{1 & -\overline\zeta \\
            0 & 1}
   \ee
   Since $h\inverse(H)=\Gamma$, we get
   \be
      \phi_\Gamma=\hat h\phi_H h
      =\smallmatr{1      & -\overline\zeta \\
             -\zeta & 1}
   \ee
   In conclusion we find
   \be
      \phi_{\Delta+\Gamma}
      =\smallmatr{2        & -1-\overline\zeta \\
             -1-\zeta & 2}
   \ee
   The assertion follows now from
   \eqnref{eqn:canonical-homs} using
   the matrices we just found.
\end{proof}

\begin{lemma}\label{lemma:curve-N}
   The divisor $H+\Delta+\Gamma-V$
   is numerically equivalent to an elliptic curve $N$.
   We have
   \be
      \phi_N
      =\matr{1        & -1-\overline\zeta \\
             -1-\zeta & 3}
   \ee
\end{lemma}

\begin{proof}
   The line bundle $H+\Delta+\Gamma-V$ has self-intersection 0
   and it has positive intersection with the ample bundle
   $\Delta+\Gamma$.
   This implies that
   its numerical class belongs to
   a sum of numerically equivalent elliptic curves.
   As its intersection with $H$ is 1, it is in fact the class of
   a single elliptic curve.
   The second assertion follows from the
   equation
   \be
      \phi_N=\phi_H+\phi_\Delta+\phi_\Gamma-\phi_V
   \ee
   upon using the
   explicit matrix representations of the maps that were
   worked out above.
\end{proof}

   The following statement can be useful in understanding the map
   $B\to\P^2$, or in the construction of a basis of $H^0(B,L)$.

\begin{lemma}\label{lemma:Fix-in-K}
   Consider the line bundle $M=V+H+\Delta+\Gamma$ on $A$.
   All nine fixed points of the automorphism
   $\sigma=(\zeta^2,\zeta^2)$ are
   contained in the kernel $K(M)$ of $\phi_M$. In other words,
   if $D\in|M|$, then
   \be
      t_x^*D \in |M| \qquad\mbox{for every $x$ in $\Fix(\sigma)$}
   \ee
\end{lemma}

\begin{proof}
%   The fixed points of $\sigma$ are the nine points
%   \be
%      p_{ij}=(\tfrac i3+\tfrac i3\zeta,\ \tfrac j3+\tfrac j3\zeta)
%      \qquad (0\le i,j\le 2)
%   \ee
   From the equation
   $\phi_M=\phi_V+\phi_H+\phi_\Delta+\phi_\Gamma$ we get
   \be
      \phi_M=\matr{3        & -1-\overline\zeta \\
                   -1-\zeta & 3}
   \ee
   and one checks that $\phi_M\cdot p_{ij}$ is contained in
   $(\Z+\Z\zeta)\times(\Z+\Z\zeta)$ for every $i$ and $j$.
\end{proof}

\paragraph{Preimages of lines.}
%   Write $p_i=\frac i3+\frac i3\zeta$ with $0\le 1\le 2$
%   for the three fixed points of $\zeta^2$ on $T$.
   As we know, the three translates $V_i=V+p_{i0}$ map to points
   in $\P^2$.
%   We would like to see the preimages on $A$
%   of the lines $\ell_{ij}$
%   through any two of them.
   We would like to see the curves on $A$ which correspond to
   the lines $\ell_{ij}$
   through any two of those points.
   As the preimage of $\ell_{ij}$ contains $V_i$ and $V_j$,
   we have $V+H+\Delta+\Gamma=V_i+V_j+N_{ij}$, where
   the residual curve $N_{ij}$ is an elliptic curve
   (by Lemma \ref{lemma:curve-N}). Its intersection numbers with the
   generators are
   \be
      N_{ij}\cdot V=3, \
      N_{ij}\cdot H=1, \
      N_{ij}\cdot \Delta=1, \
      N_{ij}\cdot \Gamma=1
   \ee
   On the other hand,
   every elliptic curve on $A$ that passes through the origin
   arises as the image of a homomorphism
   $T\to A$, $x\mapsto(ax+b\zeta x, cx+d\zeta x)$
   for suitable integers $a,b,c,d$
   (see \cite{HN:Existence}).
   Using the method from \cite[Sect.~4.2]{BS:self-product}
   one finds that the elliptic curve $N$ corresponding to
   $(a,b,c,d)=(1,1,0,1)$, i.e, the image of the map
   $x\mapsto (x+\zeta x, \zeta x)$
   has the same intersection numbers as $N_{ij}$ and
   is therefore numerically equivalent to $N_{ij}$.
   So $N_{ij}$ can be obtained from $N$ by a translation --
   and
   we determine now explicitly such a translation.
   The idea is this: We know that the divisor $V_i+V_j+N_{ij}$
   passes through all 9 points $p_{ij}$. Since $V_i$ and $V_j$
   cover 6 of them, $N_{ij}$ must pass through the remaining 3.
   Now,
   a computation shows that the intersection points of $N$ and
   $V_0$
   are $p_{00}, p_{01}, p_{02}$. This implies that
   $N=N_{12}$. The other cases are obtained via translation by
   suitable fixed points -- altogether we have
   \be
      N_{12} & = & N \\
      N_{02} & = & N+p_{10}=t^*_{-p_{10}}N=t^*_{p_{20}}N \\
      N_{01} & = & N+p_{20}=t^*_{-p_{20}}N=t^*_{p_{10}}N
   \ee

\paragraph{The images of the contracted translates.}
   We know that the 12 translates of $V,H,\Delta,\Gamma$
   by fixed-points of $\sigma$ map to points in $\P^2$. We will use the
   notation
   \be
      V_i=V+p_{i0}, \quad
      H_i=H+p_{0i}, \quad
      \Delta_i=\Delta+p_{i,2i}, \quad
      \Gamma_i=\Gamma+p_{ii}
   \ee
   for these translates (where $0\le i\le 2$) and we will
   determine the coordinates of their image points.
   Consider
   in the linear series $|L|$
   the divisors
   \be
      D_0 := V_1+V_2+N_{12} \\
      D_1 := V_0+V_2+N_{02} \\
      D_2 := V_0+V_1+N_{01}
   \ee
   We choose sections $s_i\in H^0(A,M)$ defining them
   and use these to
   define
   the map $\phi_L:B\to\P^2$.
   Clearly the vertical curves
   $V_0,V_1,V_2$ then map to the points
   \begin{equation}\label{eqn:V}
      v_0=(1:0:0),\quad
      v_1=(0:1:0),\quad
      v_\infty=(0:0:1)
   \end{equation}
   respectively.
   We will now use the projective representation
   $K(M)\to \PGL(H^0(A,M))$
   (see~\cite[Chap.~6]{BL:CAV})
   in order to determine the coordinates of the images of
   the remaining nine curves $H_i,\Delta_i,\Gamma_i$.
   By Lemma~\ref{lemma:Fix-in-K} we have
   $\Fix(\zeta^2)\subset K(M)$,
   and we know that translation by fixed points
   leaves the condition of vanishing in these points invariant.
   Therefore
   the representation restricts to
   $\Fix(\zeta^2)\to\PGL(H^0(B,L))$.
   In other words,
   each of the nine fixed points gives rise to a projective
   transformation of $\P^2$.
   Let $M_{ij}$ denote the projective transformation corresponding
   to $p_{ij}$. As the translation
   $t_{p_{10}}$ cycles the vertical translates,
   $V_0\to V_1\to V_2\to V_0$,
   we know that
   $M_{10}$ must be of the form
   \be
      M_{10}=\smallmatr{0    & 0         & \lambda_3 \\
                   \lambda_1 & 0         & 0 \\
                   0         & \lambda_2 & 0
                   }
   \ee
   with non-zero entries $\lambda_i$.
   Note now that scaling the sections $s_i$ corresponds
   to a diagonal transformation on $\P^2$. We can therefore
   scale the $s_i$
   (which leaves the coordinates in~\eqnref{eqn:V} invariant)
   in such a way that in fact
   $\lambda_1=\lambda_2=\lambda_3=1$, so that
   \be
      M_{10}=\smallmatr{0 & 0 & 1 \\
                   1      & 0 & 0 \\
                   0      & 1 & 0
                   }
   \ee
   The key is now
   the fact that
   the horizontal curves $H_0,H_1,H_2$ are fixed under $p_{10}$.
   Their coordinate vectors must therefore be eigenvectors of
   $M_{10}$. These are
   \begin{equation}\label{eqn:H}
      h_0=(1:1:1),\quad
      h_1=(1:\eps:\eps^2),\quad
      h_\infty=(1:\eps^2:\eps)
   \end{equation}
   where $\eps$ denotes a primitive third root of unity.
   (After possibly rechoosing the origin in $A$ they
   are in this order.)

   Consider now the translation $t_{p_{01}}$. It fixes
   $V_0,V_1,V_2$ and therefore $M_{01}$ is of the form
   \be
      M_{01}=\smallmatr{ \mu_1 & 0     & 0 \\
                             0 & \mu_2 & 0 \\
                             0 & 0     & \mu_3
                   }
   \ee
   with non-zero entries $\mu_i$.
   Since $t_{p_{01}}$ maps $H_0$ to $H_1$, we have in fact
   \be
      M_{01}=\smallmatr{ 1 & 0    & 0\\
                         0 & \eps & 0 \\
                         0 & 0    & \eps^2
                   }
   \ee
   after scaling $M_{01}$ if necessary.
   (Here we use that the images of the $H_i$ are given by
   the coordinates, and the order, in \eqnref{eqn:H}.)
   With this information at hand,
   we can now also determine the coordinates of the images of
   the $\Delta_i$ and $\Gamma_i$.
   First, the diagonal translates $\Delta_i$ are fixed under
   $t_{p_{11}}$, and therefore
   their
   images are given by the eigenvectors of the matrix
   \be
      M_{11}=M_{10}\cdot M_{01}
      =\smallmatr{ 0 & 0      & 1 \\
                \eps & 0      & 0 \\
                   0 & \eps^2 & 0
             }
   \ee
   Thus we get the points
   \begin{equation}\label{eqn:D}
      \delta_0=(\eps:1:1),\quad
      \delta_1=(1:\eps:1),\quad
      \delta_\infty=(1:1:\eps)
   \end{equation}
   Here the first of these points is the image of $\Delta_0$, because
   it is this point among the three which is
   collinear with the images of $V_0$ and $H_0$.
   And finally, the graph translates $\Gamma_i$ are fixed under
   $t_{p_{12}}$, which leads us to the matrix
   \be
      M_{12}=M_{01}^2\cdot M_{10}
      =\smallmatr{ 0      & 0    & 1 \\
                   \eps^2 & 0    & 0 \\
                   0      & \eps & 0
             }
   \ee
   and the coordinates
   \begin{equation}\label{eqn:G}
      \gamma_0=(\eps^2:1:1),\quad
      \gamma_1=(1:\eps^2:1),\quad
      \gamma_\infty=(1:1:\eps^2)
   \end{equation}
   The first of these points
   is the image of $\Gamma_0$ (again by collinearity with $V_0$ and $H_0$).
   Summing up, we found that
   in the chosen basis of $H^0(L)$ the
   image points of the 12 translates
   are given by
   \eqnref{eqn:V}, \eqnref{eqn:H}, \eqnref{eqn:D} and
   \eqnref{eqn:G},
   and these
   coincide with the
   points in the dual Hesse configuration in standard form.

%*****************************************************************************

\paragraph{The Roulleau-Urz\'ua configuration.}
   Let $n=3m$ for some integer $m\geq 1$. Using the group of $n$-torsion points on $A$,
   we can translate the curves $V, H, \Delta$ and $\Gamma$ to obtain an a priori count of
   $4n^4$ curves, 4 each at each of the $n^4$ $n$-torsion points.
   Since each of the divisors $V, H, \Delta$ and $\Gamma$ contain $n^2$ of the torsion points,
   and thus are their own images under translation by this subgroup,
   there are only actually $4n^4/n^2=4n^2$ curves.
   The images under $\varphi_L:B\to\P^2$
   of the proper transforms under $B\to A$ of these curves form the
   Roulleau-Urz\'ua configuration. We can now prove that these are exactly
   the Halphen cubics of order $m$.

\begin{proof}[Proof of Theorem \ref{H=RU}]
The $n$-torsion subgroup contains the order 9 subgroup consisting of the
   9 points fixed with respect to the action of $\sigma$ on $A$. Each of the 12 curves through
   these 9 points (these are the curves shown in Figure \ref{FigureOnA}) map to points of $\P^2$,
   and we found these points above. The orbits under applications of $\sigma$
   among the remaining $4(n^2-3)$ curves consist of 3 curves each. Thus
   under $\varphi_L$ these curves map 3 to 1 to cubic curves, and the
   images of the $4(n^2-3)$ curves in $A$
   are $4(n^2-3)/3=4(3m^2-1)$ cubic curves in $\P^2$.     %These are the Halphen cubics of order $m$.

   The curves in the pencils $(H_u), (\Delta_u)$ and $(\Gamma_u)$ meet
   $V_0, V_1, V_2$, so their images pass through the points $v_0, v_1, v_\infty$.
   Similarly, the curves in the pencils $(V_u), (\Delta_u)$ and $(\Gamma_u)$ meet
   $H_0, H_1, H_2$, so their images pass through the points $h_0, h_1, h_\infty$,
   and so on. All together, the images of curves in $(V_u)$ pass through all
   9 points in $\Lambda_v$, so they are members of $\CC_v$, and similarly
   the pencils $(H_u)$, $(\Delta_u)$ and $(\Gamma_u)$ map to the pencils
   $\CC_h, \CC_\delta$ and $\CC_\gamma$.

   Note that $n^2-3$ curves on $A$ come from each of the 4 pencils.
   The ones which meet $V$ come from the pencils
   $(H_u), (\Delta_u)$ and $(\Gamma_u)$, with one from each pencil meeting
   $V$ at each of the $n^2-3$ $n$-torsion points on $V$ not fixed by $\sigma$.
   Now, $V$ maps to the point $v_0\in\P^2$, and the triples of curves at each torsion point
   of $V$ thus map to curves with the same tangent direction at $v_0=(1:0:0)$.
   The tangent directions are the $(n^2-3)/3$ infinitely near images in $X$ of the
    $n$-torsion points on $V$ not fixed by $\sigma$. We also get 3 tangent directions
    corresponding to the infinitely near images of the 3 $n$-torsion points on $V$ fixed
    by $\sigma$; but these we know to be the tangent directions of
    the 3 lines of the dual Hesse configuration that pass through $v_0$,
    which are the images of the 3 exceptional divisors $E_{0i}$
    above the fixed points on $V$. So the restriction of the quotient map
    $B \rightarrow X$ to $V$ is a degree 3 morphism
    $V\rightarrow E_{v_0}$ triply ramified above the directions of the dual Hesse lines,
    and therefore the tangent directions to the Roulleau-Urzúa cubics in
    $\CC_h, \CC_\delta$ and $\CC_\gamma$ are exactly the points in
    $E_{v_0}[3m]$. So, they are indeed the Halphen cubics of order $m$
    in the pencils $\CC_h, \CC_\delta$ and $\CC_\gamma$.

    The same argument applied to the restriction of $B \rightarrow X$ to
    $H$ proves that the Roulleau-Urz\'ua cubics in the remaining pencil
    $\CC_v$ are the Halphen cubics as well.
\end{proof}

%*****************************************************************************

%*****************************************************************************

%\enlargethispage{\baselineskip}

\address{Thomas Bauer,
    Fach\-be\-reich Ma\-the\-ma\-tik und In\-for\-ma\-tik,
   Philipps-Uni\-ver\-si\-t\"at Mar\-burg,
   Hans-Meer\-wein-Stra{\ss}e,
   D-35032~Mar\-burg, Germany}
\email{tbauer@mathematik.uni-marburg.de}

\address{Brian Harbourne\\
Department of Mathematics\\
University of Nebraska\\
Lincoln, NE 68588-0130 USA}
\email{bharbourne1@unl.edu}

\address{Joaquim Ro\'e\\
Departament de Matem\`atiques\\
Universitat Aut\`onoma de Barcelona\\
08193 Bellaterra (Barcelona)\\
Spain}
\email{jroe@mat.uab.cat}

\address{Tomasz Szemberg\\
   Department of Mathematics\\
   Pedagogical University of Cracow\\
   Podchor\c a\.zych 2\\
   PL-30-084 Krak\'ow, Poland
\par
   Current Address:
   Polish Academy of Sciences\\
   Institute of Mathematics\\
   \'Sniadeckich 8\\
   PL-00-656 Warszawa, Poland
}
\email{tomasz.szemberg@gmail.com}

%*****************************************************************************

\end{document}